\documentclass[11pt]{article}
\usepackage{amsmath, pb-diagram}
\usepackage{amssymb}
\usepackage{amscd}
\date{}
\newenvironment{proof}{\par\noindent{\bf Proof \,}}{$\hfill \Box$\par\bigskip}
\newtheorem{thm}{Theorem}[section]
\newtheorem{lem}[thm]{Lemma}
\newtheorem{prop}[thm]{Proposition}
\newtheorem{cor}[thm]{Corollary}

\newtheorem{df}[thm]{Definition}

\newtheorem{ex}[thm]{Example}

\begin{document}

\title{Strongly Nil $*$-Clean Rings}

\author{Huanyin CHEN, Abdullah HARMANCI\, and  A. \c Ci\u gdem \" OZCAN\thanks{Corresponding author}}

\maketitle

\begin{abstract} A $*$-ring $R$ is called {\em strongly nil $*$-clean} if every element of $R$ is the sum of a
projection and a nilpotent element that commute with each other.
In this paper we prove that $R$ is a strongly nil $*$-clean
ring if and only if every idempotent in $R$ is a projection, $R$
is periodic, and $R/J(R)$ is Boolean. For any commutative $*$-ring
$R$ with $\mu^*=\mu$, $\eta^*=\eta\in R$, the algebraic extension $R[i]=\{ a+bi\, |\, a,b\in R,\, i^2=\mu i+\eta\,\}$ is strongly nil $*$-clean if and only if $R$
is strongly nil $*$-clean and $\mu\eta$ is nilpotent. We also
prove that a $*$-ring $R$ is commutative, strongly
nil $*$-clean and every primary ideal is maximal if and only if every element of $R$ is a projection.
\end{abstract}

\section{Introduction}

Let $R$ be an associative ring with unity. A ring $R$ is called
{\em strongly nil clean} if every element of $R$ is the sum of an
idempotent and a nilpotent that commute. These rings are first
discovered by Hirano-Tominaga-Yakub \cite{HTY} and were refered to as [E-N]-representable rings.
In \cite{D} and \cite{D1}, Diesl refers to this class as strongly nil clean and studies their properties.
Studying strongly nil cleanness is also relevent for Lie algebra. The decomposition of
matrices as in the definition of strongly nil cleanness over a
field must be the Jordan-Chevalley decomposition in Lie theory.
 An {\em involution}
of a ring $R$ is an operation $* :R \rightarrow R$ such that
$(x+y)^* = x^*+y^*$, $(xy)^* = y^*x^*$ and $(x^*)^* = x$ for all
$x, y \in R$. A ring $R$ with involution $*$ is called a {\it
$*$-ring}. An element $p$ in a $*$-ring $R$ is called a {\it
projection} if $p^2 = p = p^*$ (see \cite{B}). Recently the concept of strongly
clean rings are considered for any $*$-ring. Va\v{s} \cite{Va} calls a $*$-ring $R$ {\it strongly $*$-clean} if each of its elements is the sum
of a projection and a unit that commute with each other (see also \cite{ LZ}).

In this paper, we adapt strongly nil cleanness to $*$-rings. We
call a $*$-ring $R$ {\em strongly nil $*$-clean} if every element
of $R$ is the sum of a projection and a nilpotent element that
commute.  The paper
consists of three parts. In Section 2, we characterize the class of strongly nil $*$-clean rings on several different ways. For example, we show that a ring $R$ is a strongly nil $*$-clean ring if and
only if every idempotent in $R$ is a projection, $R$ is periodic,
and $R/J(R)$ is Boolean. In Section 3, we prove a result related
to the strongly nil $*$-cleanness of a
commutative $*$-ring and its algebraic extension. For a commutative
$*$-ring $R$ with $\mu^*=\mu$, $\eta^*=\eta\in R$, $R[i]=\{ a+bi~|~a,b\in R,$ $i^2=\mu i+\eta~\}$ is strongly nil $*$-clean if and only if
$R$ is strongly nil $*$-clean and $\mu\eta$ is nilpotent. Foster
\cite{F} introduced the concept of Boolean-like rings as a
generalization of Boolean rings in the commutative case. We adopt the concept of Boolean-like rings to rings with involution
and prove that a $*$-ring $R$ is $*$-Boolean-like if and only
if $R$ is strongly nil $*$-clean and $\alpha\beta=0$ for all
nilpotent elements $\alpha$, $\beta$ in $R$. In the last section,
we investigate submaximal ideals (\cite{S}) of strongly nil $*$-clean rings; and also define $*$-Boolean rings as $*$-rings over which every element is a projection and characterize them in terms of strongly nil $*$-cleanness.

Throughout this paper all rings are associative with unity (unless
otherwise noted). We write $J(R)$, $N(R)$ and $U(R)$ for the Jacobson radical of a
ring $R$, the set of all nilpotent elements in $R$ and the set of all units in $R$, respectively.
The ring of all polynomials in one variable over $R$ is denoted by $R[x]$.

\section{Characterization Theorems}

The main purpose of this section is to explore the structure of strongly nil $*$-clean rings. A ring $R$ is called {\it uniquely nil clean} if, for any $x\in R$, there exists a unique idempotent $e\in R$ such
that $x-e\in N(R)$ \cite{D}. If, in addition, $x$ and $e$ commute, $R$ is called {\it uniquely strongly nil clean} \cite{HTY}.
Strongly nil cleanness and uniquely strongly nil cleanness are equivalent by \cite[Theorem 3]{HTY}.
Analogously, for a $*$-ring, we define {\em uniquely strongly nil $*$-clean rings} by replacing ``idempotent" with
``projection" in the definition of uniquely strongly nil clean rings.

We will use the following lemma frequently.

\begin{lem}\label{21} {\rm \cite[Lemma 2.1]{LZ}} Let $R$ be a $*$-ring. If every idempotent in $R$ is a projection, then $R$ is abelian, i.e. every idempotent in $R$ is central.
\end{lem}

\begin{prop}\label{prop21} Let $R$ be a $*$-ring. Then the following are equivalent.
\begin{enumerate}
\item[\rm{(i)}] $R$ is strongly nil $*$-clean;
\item[\rm{(ii)}] $R$ is uniquely nil clean and every idempotent in $R$ is a projection;
\item[\rm{(iii)}] $R$ is uniquely strongly nil $*$-clean.
\end{enumerate}
\end{prop}

\begin{proof} If $R$ is strongly nil $*$-clean, then $R$ is strongly $*$-clean. For, if $x\in R$, then there exist a projection $e\in R$ and $w\in N(R)$ such that $2-x=e+w$ and $ew=we$. This gives that $x=(1-e)+(1-w)$ where $1-e$ is a projection and $1-w\in U(R)$.
If $R$ is strongly $*$-clean, then every idempotent in $R$ is a projection by \cite[Theorem 2.2]{LZ}.
By \cite[Theorem 3]{HTY}, the proof is completed.
\end{proof}

We note that the condition ``every idempotent in $R$ is a projection" in Proposition~\ref{prop21} is necessary as the following example shows.

\begin{ex}\label{ex22}  \rm Let $R=\big \{\left(\begin{array}{ll}
 0& 0\\
 0& 0
\end{array} \right), \left(\begin{array}{ll}
 1& 0\\
 0& 1
\end{array} \right),\left(\begin{array}{ll}
 1& 1\\
 0& 0
\end{array} \right),\left(\begin{array}{ll}
 0& 1\\
 0& 1
\end{array} \right)\big\}$ where $0,1\in \mathbb{Z}_2$.  Define $*: R\rightarrow R$, $\left(\begin{array}{ll}
 a& b\\
 c& d
\end{array} \right) \mapsto \left(\begin{array}{cc}
 a+b& b\\
 a+b+c+d& b+d
\end{array} \right)$. Then $R$ is a commutative $*$-ring with the usual matrix addition and multiplication. In
fact, $R$ is Boolean. Thus, for any $x\in R$, there exists a
unique idempotent $e\in R$ such that $x-e\in R$ is nilpotent.
But it is not strongly nil $*$-clean because the only projections are the trivial projections and there does not exist a projection $e$ in $R$ such that $\left(\begin{array}{ll}
 1& 1\\
 0& 0
\end{array} \right)-e$ is nilpotent.
\end{ex}

On the other hand, in \cite[Theorem 3]{HTY}, it is proved that $R$ is strongly nil clean if and only if $N(R)$ is an ideal and $R/N(R)$ is Boolean.
Also, $R$ is uniquely nil clean if and only if $R$ is abelian, $N(R)$ is an ideal and $R/N(R)$ is Boolean \cite[Theorem 4]{HTY}.  So if we adopt these results to rings with involution, immediately we have the following theorem by using Proposition~\ref{prop21}. But we give a new proof to the necessity.

\begin{thm}\label{thm23} Let $R$ be a $*$-ring. Then $R$ is strongly nil $*$-clean
if and only if
\begin{enumerate}
\item[\rm{(1)}] Every idempotent in $R$ is a projection;
\item[\rm{(2)}] $N(R)$ forms an ideal;
\item[\rm{(3)}] $R/N(R)$ is Boolean.
\end{enumerate}
\end{thm}

\begin{proof} Assume that $R$ is strongly nil $*$-clean. In view of Lemma~\ref{21}, for any $x\in R$, there exist an idempotent $g\in R$ and a
nilpotent element $v\in R$ such that $x=g+v$ and $gv=vg$. Thus, $x^2=g+(2g+v)v$, and
so $x-x^2=(1-2g-v)v\in N(R)$. Write $(x-x^2)^m=0$, and
so $x^m\in x^{m+1}R$. This shows that $R$ is strongly
$\pi$-regular. According to \cite[Theorem 3]{Ba}, $N(R)$ forms an
ideal of $R$. Further, $x-x^2\in N(R)$, and so $R/N(R)$ is
Boolean.
The converse is obvious by \cite[Theorem 3]{HTY}.\end{proof}

A ring $R$ is called {\em strongly $J$-$*$-clean} if for
any $x\in R$ there exists a unique projection $e\in R$ such that
$x-e\in J(R)$ \cite{CHO}.

\begin{lem}\label{lem27} Let $R$ be a $*$-ring. Then $R$ is strongly nil $*$-clean
if and only if
\begin{enumerate}
\item[\rm{(1)}] $R$ is strongly $J$-$*$-clean;
\item[\rm{(2)}] $J(R)$ is nil.
\end{enumerate}
\end{lem}

\begin{proof} Suppose that $R$ is strongly nil $*$-clean.
In view of Theorem~\ref{thm23}, $N(R)$ forms an ideal of $R$, and this gives that $N(R)\subseteq J(R)$.
On the other hand, for any $x\in J(R)$, there exists a projection $e\in R$ such that $x-e\in N(R)$.
Then $e=x-(x-e)\in J(R)$. This shows that $e=0$, and so $x$ is nilpotent. That is $J(R)$ is nil, and so $N(R)=J(R)$.
In view of Proposition~\ref{prop21}, we can see that there exists a unique projection $e\in R$ such that $x-e\in J(R)$.
Hence $R$ is strongly $J$-$*$-clean by \cite[Theorem 3.3]{CHO}.

Conversely, assume that (1) and (2) hold. In view of \cite[Proposition 2.1]{CHO}, $R$ is strongly $*$-clean.
Thus, $R$ is abelian. Let $x\in R$. By virtue of \cite[Theorem 3.3]{CHO}, there exist a
projection $e\in R$ and a $w\in J(R)$ such that $x=e+w$ and
$xe=ex$. As $J(R)$ is nil, $w\in R$ is nilpotent. Therefore $R$ is
strongly nil $*$-clean.\end{proof}

From Lemma~\ref{lem27} and \cite[Proposition 2.1]{CHO}, it follows that

\begin{center} $\{$strongly nil $*$-clean$\}\subset \{$strongly $J$-$*$-clean$\}\subset
\{$strongly $*$-clean$\}.$
\end{center}
The first inclusion is strict, because, for example, the power series ring $\mathbb
{Z}_2[[x]]$ is strongly $J$-$*$-clean but not strongly nil $*$-clean where $*$ is the identity involution by \cite[Example 2.5(5)]{C3}.  The second inclusion is also strict by \cite[Example 2.2(2)]{CHO}.

We should note that a strongly nil clean ring may not be strongly $J$-clean (see \cite[Example on p. 3799]{C3}).
Hence strongly nil clean and strongly nil $*$-clean classes have different behavior when compared to classes of strongly $J$-clean and strongly $J$-$*$-clean classes respectively.

\begin{lem}\label{lem28} Let $R$ be a $*$-ring. Then $R$ is strongly nil $*$-clean
if and only if
\begin{enumerate}
\item[\rm{(1)}] Every idempotent in $R$ is a projection;
\item[\rm{(2)}] $J(R)$ is nil;
\item[\rm{(3)}] $R/J(R)$ is Boolean.
\end{enumerate}
\end{lem}
\begin{proof} Assume that $(1),(2)$ and $(3)$ hold. For any $x\in
R$, $x+J(R)=x^2+J(R)$. As $J(R)$ is nil, every idempotent in $R$
lifts modulo $J(R)$. Thus, we can find an idempotent $e\in R$ such
that $x-e\in J(R)\subseteq N(R)$. By Lemma~\ref{21}, $xe=ex$, and so the result follows.
The converse is by Theorem~\ref{thm23} and Lemma~\ref{lem27}.\end{proof}

Recall that a ring $R$ is {\em periodic} if for any $x\in R$,
there exist distinct $m,n\in \mathbb{N}$ such that $x^m=x^n$. With
this information we can now prove the following.

\begin{thm}\label{thm29} Let $R$ be a $*$-ring. Then $R$ is strongly nil $*$-clean
if and only if
\begin{enumerate}
\item[\rm{(1)}] Every idempotent in $R$ is a projection;
\item[\rm{(2)}] $R$ is periodic;
\item[\rm{(3)}] $R/J(R)$ is Boolean.
\end{enumerate}
\end{thm}
\begin{proof} Suppose that $R$ is strongly nil $*$-clean. By
virtue of Lemma~\ref{lem28}, every idempotent in $R$ is a
projection and $R/J(R)$ is Boolean. For any $x\in R$, $x-x^2\in
N(R)$. Write $(x-x^2)^m=0$, and so $x^m=x^{m+1}f(x)$, where
$f(x)\in {\Bbb Z}[x]$. According to Herstein's Theorem (cf.
~\cite[Proposition 2]{Cha}), $R$ is periodic. Conversely, $J(R)$
is nil as $R$ is periodic. Therefore the proof is completed by
Lemma~\ref{lem28}.\end{proof}

\begin{prop}\label{prop210} A $*$-ring $R$ is strongly nil $*$-clean if and only if
\begin{enumerate}
\item[\rm{(1)}] $R$ is strongly $*$-clean;
\item[\rm{(2)}] $N(R)=\{ x\in R~|~1-x\in U(R)\}$.
\end{enumerate}
\end{prop}
\begin{proof} Suppose that $R$ is strongly nil $*$-clean. By the proof of Lemma~\ref{lem27}, $N(R)=J(R)$.
Since $R$ is strongly $J$-$*$-clean, $N(R)=\{ x\in R~|~1-x\in
U(R)\}$ by \cite[Theorem 3.5]{CHO}.

Conversely, assume that $(1)$ and $(2)$ hold. Let $a\in R$. Then
we can find a projection $e\in R$ such that $(a-1)-e\in U(R)$ and
$e(a-1)=(a-1)e$. That is, $(1-a)+e\in U(R)$. As $1-(a-e)\in U(R)$,
by hypothesis, $a-e\in N(R)$. In addition, $ea=ae$. Accordingly, $R$ is
strongly nil $*$-clean.\end{proof}

Let $R$ be a $*$-ring. Define $*: R[x]/(x^n)\to R[x]/(x^n)$ by $a_0+a_1x+\cdots +a_{n-1}x^{n-1}+(x^n)\mapsto a_0^*+a_1^*x+\cdots
+a_{n-1}^*x^{n-1}+(x^n)$. Then $R[x]/(x^n)$ is a $*$-ring (cf. \cite{LZ}).

\begin{cor}\label{cor211} Let $R$ be a $*$-ring. Then $R$ is strongly nil $*$-clean
if and only if so is $R[x]/(x^n)$ $(n\geq 1)$.
\end{cor}
\begin{proof} One direction is obvious. Conversely, assume that
$R$ is strongly nil $*$-clean. Clearly, $N\big(R[x]/(x^n)\big)=\{
a_0+a_1x+\cdots +a_{n-1}x^{n-1}+(x^n)~|~a_0\in N(R), a_1,\cdots
,a_{n-1}\in R\}$. In view of Proposition~\ref{prop210}, $N\big(R[x]/(x^n)\big)=\{
a_0+a_1x+\cdots +a_{n-1}x^{n-1}+(x^n)~|~1-a_0\in U(R), a_1,\cdots
,a_{n-1}\in R\}$. Also note that $R$ is abelian. Thus, it can be easily seen that every element in $R[x]/(x^n)$ can be written as the sum of a projection and a nilpotent element that commute.\end{proof}

\section{Algebraic Extensions}

Let $R$ be a commutative $*$-ring, and let $\mu, \eta \in R$ with $\mu^*=\mu$ and $\eta^*=\eta$. Let $R[i]=\{ a+bi~|~a,b\in R,$ $i^2=\mu
i+\eta~\}$. Then $R[i]$ is a $*$-ring, where the involution is $*: R[i]\to R[i]$, $a+bi\mapsto a^*+b^*i$. The aim of this section is to
explore the algebraic extensions of a
strongly nil $*$-clean ring.

\begin{prop}\label{thm31} Let $R$ be a commutative $*$-ring with $\mu^*=\mu$, $\eta^*=\eta\in R$. Then $R[i]$ is strongly nil $*$-clean if and only if
\begin{itemize}
\item[{\rm(1)}] $R$ is strongly nil $*$-clean;
\item[{\rm(2)}] $\mu\eta$ is nilpotent.
\end{itemize}
\end{prop}
\begin{proof} Suppose that $R[i]$ is strongly nil $*$-clean. Then every idempotent in $R$ is a projection.
Since $R$ is commutative, $N(R)$ forms an ideal. For any $a\in R$, we see
that $a-a^2\in N\big(R[i]\big)$, and so $a-a^2\in N(R)$. Thus, $R/N(R)$ is
Boolean. Therefore $R$ is strongly nil $*$-clean by
Theorem~\ref{thm23}. As $R[i]/N\big(R[i]\big)$ is Boolean,
$i-i^2\in N\big(R[i]\big)$. This shows that $\eta+(\mu-1)i\in
N\big(R[i]\big)$, and so $\mu\eta+(\mu^2-\mu)i\in
N\big(R[i]\big)$. As $\mu^2-\mu\in N(R)$, we see that $\mu\eta\in
N\big(R[i]\big)$. Thus, $\mu\eta$ is nilpotent.

Conversely, assume that $(1)$, $(2)$ hold. As $R$ is commutative,
$N\big(R[i]\big)$ forms an ideal of $R[i]$. Let $a+bi\in R[i]$ be
an idempotent. Then we can find projections $e,f\in R$ and
nilpotent elements $u,v\in R$ such that $a=e+u$, $b=f+w$. Then
$a-a^*, b-b^*\in N(R)$. This shows that
$(a+bi)-(a+bi)^*=(a-a^*)+(b-b^*)i\in N\big(R[i]\big)$. As $a+bi, (a+bi)^*\in R[i]$ are idempotents, we see that
$\big((a+bi)-(a+bi)^*\big)^3=\big((a+bi)-2(a+bi)(a+bi)^*+(a+bi)^*\big)\big((a+bi)-(a+bi)^*\big)=(a+bi)-(a+bi)^*$. Hence,
$\big((a+bi)-(a+bi)^*\big)\big(1-((a+bi)-(a+bi)^*)^2\big)=0$, therefore
$(a+bi)-(a+bi)^*=0$. That is, $a+bi\in R[i]$ is a projection.

Since $R$ is strongly nil $*$-clean, it follows from Theorem
\ref{thm23} that $2-2^2\in N(R)$, and so $2\in N(R)$. For any $a+bi\in
R[i]$, it is easy to verify that
$$\begin{array}{lcl}
(a+bi)-(a+bi)^2&=&(a-a^2)-2abi+bi-b^2i^2\\
&\equiv &-b^2\eta+(b-b^2\mu)i\\
&\equiv &-b\eta+b(1-\mu)i~\big(mod~N\big(R[i]\big)\big).
\end{array}$$ This shows that $$\begin{array}{lcl}
\big((a+bi)-(a+bi)^2\big)^2&\equiv &b^2\eta^2-2b^2\eta(1-\mu)i+b^2(1-\mu)^2i^2\\
&\equiv &b\eta^2+b(1-\mu)^2(\mu i+\eta)\\
&\equiv &b\eta+b(1-\mu)(\mu i+\eta)\\
&\equiv &2b\eta-b\mu\eta+b(\mu-\mu^2)i\\
&\equiv &-b\mu\eta\\
&\equiv &0~\big(mod~N\big(R[i]\big)\big).\end{array}$$ Hence,
$(a+bi)-(a+bi)^2\in N(R[i])$. That is,
$R[i]/N\big(R[i]\big)$ is Boolean. According to
Theorem~\ref{thm23}, we complete the proof.\end{proof}

As an immediate consequence, we deduce that a commutative $*$-ring
$R$ is strongly nil $*$-clean if and only if so is $R[i]$ where
$i^2=-1$.

\bigskip

We now consider a subclass of strongly nil $*$-clean rings consisting of rings which we call $*$-Boolean-like. First recall that a ring $R$ is called {\em Boolean-like} \cite{F} if it is commutative with unit and is of characteristic 2 with $ab(1+a)(1+b)=0$ for every $a,b\in R$. Any Boolean ring is clearly a Boolean-like ring but not conversely (see \cite{F}). Any Boolean-like ring is uniquely nil clean by \cite[Theorem 17]{F}. Also, $R$ is Boolean-like if and only if (1) $R$ is commutative ring with unit; (2) It is of characteristic $2$; (3) It is nil clean; (4) $ab=0$ for every nilpotent element $a,b$ in $R$ \cite[Theorem 19]{F}.

\begin{df} {\rm A $*$-ring
$R$ is said to be {\it $*$-Boolean-like} provided that every
idempotent in $R$ is a projection and $(a-a^2)(b-b^2)=0$ for all
$a,b\in R$.}
\end{df}

By using the following theorem, we can see that $*$-Boolean-like rings are commutative rings. For, let $x,y\in R$. In view of Theorem~\ref{thm38}, $x-e$ and $y-f$ are nilpotent for some projections $e,f\in R$. Again by Theorem~\ref{thm38}, $(x-e)(y-f)=0=(y-f)(x-e)$. Since $R$ is abelian, it follows that $xy=yx$. Hence $R$ is commutative.\\

The following is an example of a $*$-Boolean-like ring.

\begin{ex} {\rm Let $R=\{
\left(
\begin{array}{cc}
a&b\\
c&a
\end{array}
\right)
~|~a,b,c\in \mathbb{Z}_2\}$. Define
$
\left(
\begin{array}{cc}
a&b\\
c&a
\end{array}
\right)+\left(
\begin{array}{cc}
a'&b'\\
c'&a'
\end{array}
\right)=\left(
\begin{array}{cc}
a+a'&b+b'\\
c+c'&a+a'
\end{array}
\right),$ $
\left(
\begin{array}{cc}
a&b\\
c&a
\end{array}
\right)\left(
\begin{array}{cc}
a'&b'\\
c'&a'
\end{array}
\right)=\left(
\begin{array}{cc}
aa'&ab'+ba'\\
ca'+ac'&aa'
\end{array}
\right)$ and $
*: R\to R,$ $\left(
\begin{array}{cc}
a&b\\
c&a
\end{array}
\right)\mapsto \left(
\begin{array}{cc}
a&c\\
b&a
\end{array}
\right).
$ Then $R$ is a $*$-ring. Let $\left(
\begin{array}{cc}
a&b\\
c&a
\end{array}
\right)\in R$ be an idempotent. Then $a=a^2$ and $(2a-1)b=(2a-1)c=0$. As $(2a-1)^2=1$, we see that $b=c=0$, and so the set of all
idempotents in $R$ is $\{ \left(
\begin{array}{cc}
0&0\\
0&0
\end{array}
\right), \left(
\begin{array}{cc}
1&0\\
0&1
\end{array}
\right)\}$. Thus, every idempotent in $R$ is a projection. For any
$A,B\in R$, we see that $(A-A^2)(B-B^2)=\left(
\begin{array}{cc}
0&*\\
*&0
\end{array}
\right)\left(
\begin{array}{cc}
0&*\\
*&0
\end{array}
\right)=0$. Therefore $R$ is $*$-Boolean-like.}
\end{ex}

\begin{thm}\label{thm38} Let $R$ be a $*$-ring. Then $R$ is $*$-Boolean-like if and only if
\begin{itemize}
\item[{\rm(1)}] $R$ is strongly nil $*$-clean;
\item[{\rm(2)}] $\alpha\beta=0$ for all nilpotent elements $\alpha,\beta\in R$.
\end{itemize}
\end{thm}
\begin{proof} Suppose that $R$ is $*$-Boolean-like. Then every idempotent in
$R$ is a projection; hence, $R$ is abelian. For any $a\in R$,
$(a-a^2)^2=0$, and so $a^2=a^3f(a)$ for some $f(t)\in {\Bbb
Z}[t]$. This implies that $R$ is strongly $\pi$-regular, and so it
is $\pi$-regular. It follows from \cite[Theorem 3]{Ba} that $N(R)$
forms an ideal. Further, $a-a^2\in N(R)$. Therefore $R/N(R)$ is Boolean.
According to Theorem~\ref{thm23}, $R$ is strongly nil $*$-clean.
For any nilpotent elements $\alpha,\beta\in R$, we can find some
$m,n\in {\Bbb N}$ such that $\alpha^m=\beta^n=0$. Since
$\alpha^2=\alpha^3g(\alpha)$ for some $g(t)\in {\Bbb Z}[t]$, $\alpha^2=0$. Likewise, $\beta^2=0$. This shows that
$\alpha\beta=(\alpha-\alpha^2)(\beta-\beta^2)=0$.

Conversely, assume that $(1)$ and $(2)$ hold.
By Theorem~\ref{thm23}, every idempotent is a projection, and for any $a\in R$, $a-a^2$ is nilpotent. Hence
for any $a,b\in R$, $(a-a^2)(b-b^2)=0$. Therefore $R$ is $*$-Boolean-like.
\end{proof}

\begin{cor}\label{cor39} Let $R$ be a commutative $*$-ring with $\mu^*=\mu, \eta^*=\eta\in R$. If $\mu\in U(R)$, then
$R[i]$ is $*$-Boolean-like if and only if
\begin{itemize}
\item[{\rm(1)}] $R$ is $*$-Boolean-like;
\item[{\rm(2)}] $\eta$ is nilpotent.
\end{itemize}
\end{cor}

\begin{proof} If $R[i]$ is $*$-Boolean-like, then $R$ is $*$-Boolean-like. Also $\mu\eta\in R$ is nilpotent by
Proposition~\ref{thm31} and Theorem~\ref{thm38}. Since $\mu$ is unit and $N(R)$ is an ideal, $\eta$ is nilpotent.

Conversely, assume that $(1)$ and $(2)$ hold. Then $R[i]$ is strongly
nil-$*$-clean by Proposition~\ref{thm31}. In addition, $2\in N(R)$.
Let $a+bi\in R[i]$ be nilpotent. We claim that $a$ and $b$ are nilpotent. As $2\in N(R)$, we can find some
$n\in {\Bbb N}$ such that $(a+bi)^{2n}=a^{2n}+b^{2n}i^{2n}=0$. We
claim that $i^{2n}=ci+d$ for some $c\in U(R), d\in N(R)$. This is
true for $n=1$ by hypothesis. Assume that this holds for
$n=k (k\geq 1)$. Write $i^{2k}=\alpha i+\beta$ with $\alpha\in
U(R),\beta\in N(R)$. Now assume that $n=k+1$. Then
$i^{2(k+1)}=i^{2k}(\mu
i+\eta)=(\alpha\mu^2+\alpha\eta+\beta\mu)i+(\alpha\mu\eta+\beta\eta)$
with $\alpha\mu^2+\alpha\eta+\beta\mu\in
U(R),\alpha\mu\eta+\beta\eta\in N(R)$. Therefore
$(a+bi)^{2n}=a^{2n}+b^{2n}i^{2n}=a^{2n}+b^{2n}(ci+d)=(a^{2n}+b^{2n}d)+b^{2n}ci$.
This implies that $a^{2n}+b^{2n}d=0=b^{2n}c$. As $c\in U(R)$, we
get $b^{2n}=a^{2n}=0$. That is, $a,b\in R$ are nilpotent. Now let
$\alpha+\beta i\in R[i]$ be an another nilpotent element. Similarly, $\alpha,\beta\in
R$ are nilpotent. Thus, $(a+bi)(\alpha+\beta
i)=(a\alpha-b\beta)+(a\beta+b\beta)i=0$. By
Theorem~\ref{thm38}, $R[i]$ is $*$-Boolean-like.\end{proof}


\begin{ex} {\rm Let $R$ be the ring $${\tiny \{ \left(
\begin{array}{cc}
0&0\\
0&0
\end{array}
\right), \left(
\begin{array}{cc}
1&0\\
0&1
\end{array}
\right), \left(
\begin{array}{cc}
0&1\\
1&0
\end{array}
\right), \left(
\begin{array}{cc}
1&1\\
0&0
\end{array}
\right), \left(
\begin{array}{cc}
0&0\\
1&1
\end{array}
\right), \left(
\begin{array}{cc}
1&0\\
1&0
\end{array}
\right),\\
 \left(
\begin{array}{cc}
0&1\\
0&1
\end{array}
\right), \left(
\begin{array}{cc}
1&1\\
1&1
\end{array}
\right)\},}$$ where $0,1\in {\Bbb Z}_2.$ Define $*: R\to R,
A\mapsto A^T$, the transpose of $A$. Then $R$ is a $*$-ring in
which $(a-a^2)(b-b^2)=0$ for all $a,b\in R$. Further,
$\alpha\beta=0$ for all nilpotent elements $\alpha,\beta\in R$. But $R$ is
not $*$-Boolean-like. }
\end{ex}

We end this section with an example showing that strongly nil clean ring need not be strongly nil $*$-clean.

\begin{ex} \rm Consider the
ring
$$R=\big \{ \left(
\begin{array}{cc}
a&2b\\
0&c
\end{array}
\right) ~|~a,b,c\in {\Bbb Z}_4\big \}.$$ Then for any $x,y\in R$,
$(x-x^2)(y-y^2)=0$. Obviously, $R$ is not commutative. This implies that $R$ is
not a $*$-Boolean-like ring for any involution $*$. Accordingly, $R$ is not strongly nil $*$-clean for any involution $*$;
otherwise, every idempotent in $R$ is a projection, a
contradiction (see Lemma~\ref{21}).
We can also consider the involution $*: R\rightarrow R$, $\tiny \left(
\begin{array}{cc}
a&2b\\
0&c
\end{array}
\right)\mapsto
\left(
\begin{array}{cc}
c& -2b\\
0&a
\end{array}
\right)$ and the idempotent $\tiny \left(
\begin{array}{cc}
0&0\\
0&1
\end{array}
\right)$ which is not a projection.
On the other hand, since $(x-x^2)^2=0$ and so $x-x^2 \in N(R)$ for all $x\in R$, we get that $R$ is strongly nil clean by \cite[Theorem 3]{HTY}.
\end{ex}

\section{Submaximal Ideals and $*$-Boolean Rings}
An ideal $I$ of a ring $R$ is called a {\em submaximal ideal} if $I$ is
covered by a maximal ideal of $R$. That is, there exists a maximal
ideal $J$ of $R$ such that $I\subsetneqq J\subsetneqq R$ and for
any ideal $K$ of $R$ such that $I\subseteq K\subseteq J$ then we
have $I=K$ or $K=J$. This concept was firstly introduced to
study Boolean-like rings (cf. ~\cite{S}).

A $*$-ring $R$ is called a {\em $*$-Boolean ring} if every element of $R$ is a projection.

The purpose of this section is to characterize submaximal ideals of strongly nil $*$-clean rings, and $*$-Boolean rings by means of strongly nil $*$-cleanness. We begin with the following lemma.

\begin{lem}\label{lem41} Let $R$ be strongly nil $*$-clean. Then an ideal $M$ of $R$ is maximal if and only
if
\begin{itemize}
\item[{\rm(1)}] $M$ is prime;
\item[{\rm(2)}] For any $a\in R, n\geq 1$, $a^n\in M$ implies that $a\in M$.
\end{itemize}
\end{lem}
\begin{proof} Suppose that $M$ is maximal. Obviously, $M$ is prime. Let $a\in R$ and $a^n\in M$. If
$a\not\in M$, $RaR+M=R$. Thus,
$\overline{R}\overline{a}\overline{R}=\overline{R}$ where $\overline{R}= R/M$ and $\overline{a}=a+M$. Clearly, $R$
is an abelian clean ring, and so it is an exchange ring by \cite[Theorem 17.2.2]{Ch}.
This implies that $R/M$ is an abelian exchange ring. As in the proof of
\cite[Proposition 17.1.9]{Ch}, there exists an idempotent $e+M\in R/M$ such that
$\overline{R}(e+M)\overline{R}=\overline{R}$ and
$e+M\in \overline{a}\overline{R}$. Thus,
$1-e\in M$. Hence, $1-ar\in M$ for some $r\in R$.
This implies that $a^{n-1}-a^nr\in M$, and so $a^{n-1}\in M$. By
iteration of this process, we see that $a\in M$, as required.

Conversely, assume that $(1)$ and $(2)$ hold. Assume that $M$ is
not maximal. Then we can find a maximal ideal $I$ of $R$ such that
$M\subsetneqq I\subsetneqq R$. Choose $a\in I$ while $a\not\in M$.
By hypothesis, there exist an idempotent $e\in R$ and a nilpotent
$u\in R$ such that $a=e+u$. Write $u^m=0$. Then $u^m\in M$. By
hypothesis, $u\in M$. This shows that $e\not\in M$. Clearly, $R$
is abelian. Thus $eR(1-e)\subseteq M$. As $M$ is prime, we deduce
that $1-e\in M$. As a result, $1-a=(1-e)-u\in M$, and so
$1=(1-a)+a\in I$. This gives a contradiction. Therefore $M$ is
maximal.\end{proof}

Let $R$ be a strongly nil $*$-clean ring, and let $x\in R$. Then
there exists a unique projection $e\in R$ such that $x-e\in N(R)$.
We denote $e$ by $x_P$ and $x-e$ by $x_N$.

\begin{lem}\label{lem42} Let $I$ be an ideal of a strongly nil $*$-clean ring $R$, and let $x\in R$ be such
that $x\not\in I$. If $x_P\not\in I$, then there exists a maximal
ideal $J$ of $R$ such that $I\subseteq J$ and $x\not\in J$.
\end{lem}
\begin{proof} Let $\Omega=\{ K~|~I\subseteq K, x_P\not\in K\}$. Then
$\Omega\neq \emptyset$. Given $K_1\subseteq K_2\subseteq \cdots $
in $\Omega$, we set $Q=\bigcup\limits_{i=1}^{\infty}K_i$. Then $Q$
is an ideal of $R$. If $Q\not\in \Omega$, then $x_P\in Q$, and so
$x_P\in K_i$ for some $i$. This gives a contradiction. Thus,
$\Omega$ is inductive. By using Zorn's Lemma, there exists an
ideal $J$ of $R$ which is maximal in $\Omega$. Let $a,b\in R$ such
that $a,b\not\in J$. By the maximality of $J$, we see that $RaR+J,
RbR+J\not\in \Omega$. This shows that $x_P\in
\big(RaR+J\big)\bigcap \big(RbR+J\big)$. Hence, $x_P=x_P^{2}\in RaRbR+J$. This yields that $aRb\not\in J$;
otherwise, $x_P\in J$, a contradiction. Hence, $J$ is prime.
Assume that $J$ is not maximal. Then we can find a maximal ideal
$M$ of $R$ such that $J\subsetneqq M\subsetneqq R$. Clearly, $R$
is abelian. By the maximality, we see that $x_P\in M$, and so
$1-x_P\not\in M$. This implies that $1-x_P\not\in J$. As
$x_PR(1-x_P)=0\subseteq J$, we have that $x_P\in J$, a
contradiction. Therefore $J$ is a maximal ideal, as
asserted.\end{proof}

\begin{prop}\label{thm43} Let $R$ be strongly nil $*$-clean. Then the intersection of two maximal ideals is submaximal and it is covered by each of these
two maximal ideals. Further, there is no other maximal ideals
containing it.
\end{prop}
\begin{proof} Let $I_1$ and $I_2$ be two distinct maximal ideals
of $R$. Then $I_1\bigcap I_2\subsetneqq I_1$. Suppose $I_1\bigcap
I_2\subseteq J\subsetneqq I_1$. Then we can find some $x\in I_1$
while $x\not\in J$. Write $x_N^n=0$. Then $x_N^n\in I_1$. In light
of Lemma~\ref{lem41}, $x_N\in I_1$. Likewise, $x_N\in I_2$. Thus,
$x_N\in I_1\bigcap I_2\subseteq J$. This shows that $x_P\not\in
J$. By virtue of Lemma~\ref{lem42}, there exists a maximal ideal
$M$ of $R$ such that $J\subseteq M$ and $x\not\in M$. Hence,
$I_1\bigcap I_2\subseteq M$ and $I_1\neq M$. If $I_2\neq M$, then
$I_2+M=R$. Write $t+y=1$ with $t\in I_2, y\in M$. Then for any
$z\in I_1$, $z=zt+zy\in I_1\bigcap I_2+M=M$, and so $I_1=M$. This
gives a contradiction. Thus $I_2=M$, and then $J\subseteq
M\subseteq I_2$. As a result, $J\subseteq I_1\bigcap I_2$, and so
$I_1\bigcap I_2=J$. Therefore $I_1\bigcap I_2$ is a submaximal
ideal of $R$. We claim that $I_1\bigcap I_2$ is semiprime. If $K^2\subseteq I_1\bigcap I_2$, then for any $a\in
K$, we see that $a^2\in I_1\bigcap I_2$. In view of
Lemma~\ref{lem41}, $a\in I_1\bigcap I_2$. This implies that
$K\subseteq I_1\bigcap I_2$. Hence, $I_1\bigcap I_2$ is semiprime.
Therefore $I_1\bigcap I_2$ is the intersection of maximal ideals
containing $I_1\bigcap I_2$. Assume that $K$ is a maximal ideal of
$R$ such that $I_1\bigcap I_2\subseteq K$. If $K\neq I_1, I_2$,
then $I_1+K=I_2+K=R$. This implies that $I_1\bigcap I_2+K=R$, and
so $K=R$, a contradiction. Thus, $K=I_1$ or $K=I_2$, and so the
proof is completed.\end{proof}

We call a local ring $R$ {\em absolutely local} provided that for
any $0\neq x\in J(R)$, $J(R)=RxR$.

\begin{cor}\label{cor44} Let $R$ be strongly nil $*$-clean, and let $I$ be an ideal of $R$. Then
$I$ is a submaximal ideal if and only if $R/I$ is Boolean with
four elements or $R/I$ is absolutely local.
\end{cor}
\begin{proof} Let $I$ be a submaximal ideal of $R$.

Case I. $I$ is contained in more than a maximal ideal. Then
$I$ is contained in two distinct maximal ideals of $R$. Since $I$
is submaximal, there exists a maximal ideal $J$ of $R$ such that
$I$ is covered by $J$. Thus, we have a maximal ideal $J'$ such
that $J'\neq J$ and $I\subsetneqq J'$. Hence, $I\subseteq J\bigcap
J'\subseteq J$. Clearly, $J\bigcap J'\neq J$ as $J+J'=R$, and so
$I=J\bigcap J'$. In view of Proposition~\ref{thm43}, there is no
maximal ideal containing $I$ except for $J$ and $J'$. This shows
that $R/I$ has only two maximal ideals covering $\{ 0+I\}$. For any
$a\in R$, it follows from Theorem~\ref{thm23} that $a-a^2\in R$
is nilpotent. Write $(a-a^2)^n=0$. Then $(a-a^2)^n\in J$.
According to Lemma~\ref{lem41}, $a-a^2\in J$. Likewise, $a-a^2\in
J'$. Thus, $a-a^2\in J\bigcap J'$, and so $a-a^2\in I$. This shows
that $R/I$ is Boolean. Therefore $R/I$ is Boolean with four
elements.

Case II. Suppose that $I$ is contained in only one maximal ideal
$J$ of $R$. Then $R/I$ has only one maximal ideal $J/I$. Clearly,
$R$ is an abelian exchange ring, and then so is $R/I$. Let
$\overline{e}\in R/I$ be a nontrivial idempotent. Then $I\subseteq
I+ReR\subseteq J$ or $I+ReR=R$. Likewise, $I\subseteq
I+R(1-e)R\subseteq J$ or $I+R(1-e)R=R$. This shows that $I+ReR=R$
or $I+R(1-e)R=R$ Thus, $(R/I)(e+I)(R/I)=R/I$ or
$(R/I)(1-e+I)(R/I)=R/I$, a contradiction. Therefore all
idempotents in $R/I$ are trivial. It follows from \cite[Lemma
17.2.1]{Ch} that $R/I$ is local. For any $\overline{0}\neq
\overline{x}\in J/I$, we see that $0\neq I\subseteq RxR\subseteq
J$. As $I$ is submaximal, we deduce that $J=RxR$. Therefore $R$ is
absolutely local.

Conversely, assume that $R/I$ is Boolean with four elements. Then
$R/I$ has precisely two maximal ideals covering $\{ 0+I\}$, and so
$R$ has precisely two maximal ideals covering $I$. Thus, we have a
maximal ideal $J$ such that $I\subsetneqq J$. If $I\subseteq
K\subseteq J$. Then $K=I$ or $K$ is maximal, and so $K=J$.
Consequently, $I$ is submaximal. Assume that $R/I$ is absolutely
local. Then $R/I$ has a uniquely maximal ideal $J/I$. Hence, $J$
is a maximal ideal of $R$ such that $I\subsetneqq J$. Assume that
$I\subsetneqq K\subseteq J$. Choose $a\in K$ while $a\not\in I$.
Then $J=RaR\subseteq K$, and so $K=J$. Therefore $I$ is
submaximal, as required.
\end{proof}

\begin{cor}\label{cor45} Let $R$ be strongly nil $*$-clean. If $I_1$ and $I_2$ are distinct maximal ideals of $R$,
then $R/(I_1\bigcap I_2)$ is Boolean.
\end{cor}
\begin{proof} Since $I_1/(I_1\bigcap I_2)$ and $I_2/(I_1\bigcap
I_2)$ are distinct maximal ideals, $R/(I_1\bigcap I_2)$ is not local. In
view of Proposition~\ref{thm43}, $I_1\bigcap I_2$ is a submaximal
ideal of $R$. Therefore we complete the proof from
Corollary~\ref{cor44}.\end{proof}

Recall that an ideal $I$ of a commutative ring $R$ is {\em primary}
provided that for any $x,y\in R$, $xy\in I$ implies that $x\in I$
or $y^n\in I$ for some $n\in \mathbb{N}$. Clearly, every maximal
ideal of a commutative ring is primary. We end this article by
giving the relation between strongly nil $*$-clean rings and
$*$-Boolean rings.

\begin{lem}\label{lem46} Let $R$ be a commutative strongly nil $*$-clean ring. Then the intersection of all
primary ideals of $R$ is zero.\end{lem}
\begin{proof} Let $a$ be in the intersection of all
primary ideal of $R$. Assume that $a\neq 0$. Let $\Omega=\{
I~|~I~\mbox{is an ideal of}~R~\mbox{such that}~a\not\in I\}$. Then
$\Omega\neq \emptyset$ as $0\in \Omega$. Given any ideals
$I_1\subseteq I_2\subseteq \cdots $ in $\Omega$, we set
$M=\bigcup\limits_{i=1}^{\infty}I_i$. Then $M\in \Omega$. Thus,
$\Omega$ is inductive. By using Zorn's Lemma, we can find an ideal
$Q$ which is maximal in $\Omega$. It will suffice to show that $Q$
is primary. If not, we can find some $x,y\in R$ such that $xy\in
Q$, but $x\not\in Q$ and $y^n\not\in Q$ for any $n\in \mathbb{N}$.
This shows that $a\in Q+(x)$, and so $a=b+cx$ for some $b\in Q,
c\in R$. Since $R$ is strongly nil $*$-clean, it follows from
Theorem~\ref{thm29} that there are some distinct $k,l\in
\mathbb{N}$ such that $y^k=y^l$. Say $k>l$. Then $y^l=y^k=y^{l+1}y^{k-l-1}=y^lyy^{k-l-1}=y^{l+2}y^{2(k-l-1)}=\cdots =y^{2l}y^{l(k-l-1)}$. Hence, $y^{l(k-l)}=y^{l}(y^{l(k-l-1)})=y^{2l}y^{2l(k-l-1)}=\big(y^{l(k-l)}\big)^2$. Choose $s=l(k-l)$.
Then $y^s$ is an idempotent. Write $y=y_P+y_N$.
Then $y^s-y_P=(y_P+y_N)^s-y_P=y_N\big(sy_P+\cdots +y_N^{s-1}\big)\in
N(R)$. As $R$ is a commutative ring, we see that
$(y^s-y_P)^3=y^s-y_P$. This implies that $y^s=y_P$. Since $xy\in
Q$, we show that $xy^s\in Q$, and so $xy_P\in Q$. It follows from
$a=b+cx$ that $ay_P=by_P+cxy_P\in Q$. Clearly, $y^s\not\in Q$, and
so $a\in Q+(y_P)$. Write $a=d+ry_P$ for some $d\in Q, r\in R$. We
see that $ay_P=dy_P+ry_P$, and so $ry_P\in Q$. This implies that
$a\in Q$, a contradiction. Therefore $Q$ is primary, a contradiction.
Consequently, the intersection of all primary ideal
of $R$ is zero.
\end{proof}

\begin{thm}\label{thm47} Let $R$ be a $*$-ring. Then $R$ is a $*$-Boolean ring if and only if
\begin{itemize}
\item[{\rm(1)}] $R$ is commutative;
\item[{\rm(2)}] Every primary ideal of $R$ is maximal;
\item[{\rm(3)}] $R$ is strongly nil $*$-clean.
\end{itemize}
\end{thm}
\begin{proof} Suppose that $R$ is a $*$-Boolean ring. Clearly, $R$ is a commutative strongly nil $*$-clean
ring. Let $I$ be a primary ideal of $R$. If $I$ is not maximal,
then there exists a maximal ideal $M$ such that $I\subsetneqq
M\subsetneqq R$. Choose $x\in M$ while $x\not\in I$. As $x$ is an
idempotent, we see that $xR(1-x)\subseteq I$, and so $(1-x)^m\in
I\subset M$ for some $m\in \mathbb{N}$. Thus, $1-x\in M$. This
implies that $1=x+(1-x)\in M$, a contradiction. Therefore $I$ is
maximal, as required.

Conversely, assume that $(1),(2)$ and $(3)$ hold. Clearly, every
maximal ideal of $R$ is primary, and so $J(R)=\bigcap \{ P~|~ P~
\mbox{is primary}\}$. In view of Lemma~\ref{lem46}, $J(R)=0$.
Hence every element is a projection i.e. $R$ is $*$-Boolean. \end{proof}

\begin{cor}\label{cor48} A ring $R$ is a Boolean ring if and only if
\begin{itemize}
\item[{\rm(1)}] $R$ is commutative;
\item[{\rm(2)}] Every primary ideal of $R$ is maximal;
\item[{\rm(3)}] $R$ is strongly nil clean.
\end{itemize}
\end{cor}
\begin{proof} Choose the involution as the identity. Then the result follows from Theorem~\ref{thm47}. \end{proof}
{\bf Acknowlegement} The authors are indepted to the referee for his/her valuable comments. H. Chen is thankful for
the support by the Natural Science Foundation of
Zhejiang Province (LY13A010019).

{\bf Huanyin Chen},\\
Department of Mathematics, \\ Hangzhou
Normal University, \\Hangzhou
310036, China\\
{\bf  e-mail}: huanyinchen@yahoo.cn\\\\
{\bf  Abdullah Harmanc\i \, and A. \c Ci\u gdem \" Ozcan}, \\
Hacettepe University,\\ Department of
Mathematics, \\
 06800 Beytepe Ankara, Turkey
\\ {\bf e-mails}: abdullahharmanci@gmail.com;\\
ozcan@hacettepe.edu.tr
\end{document}